\documentclass[review]{elsarticle}
\usepackage{amsmath,amsfonts,amsthm,amssymb,verbatim,graphicx,subfigure,color}
\usepackage{mathtools}
\usepackage{epsfig}
\usepackage{makeidx}
\usepackage{latexsym}
\usepackage{enumerate}
\usepackage{todonotes}
\usepackage{algorithm}
\usepackage{algorithmic}
\usepackage{hyperref}
\usepackage{epsfig}

\newtheorem{thm}{Theorem}

\newtheorem{lem}{Lemma}
\newtheorem{cor}{Corollary}
\newtheorem{rmk}{Remark}
\usepackage{hyperref}
\newcommand{\eigenmin}{x_{\rm{min}}}
\newcommand{\eigenmax}{x_{\rm{max}}}
\journal{Linear Algebra and its Applications}

\newcommand{\Si}{\mathcal{S}}
\newcommand{\Le}{\mathcal{L}}
\begin{document}

\begin{frontmatter}

\title{On the spectral radius of graphs: nonregular distance-hereditary graphs with given edge-connectivity, graphs with tree-width $k$ and block graphs with prescribed independence number $\alpha$\tnoteref{mytitlenote}}
\tnotetext[mytitlenote]{Cristian M. Conde acknowledges partial support from ANPCyT PICT 2017-2522. Ezequiel Dratman and Luciano N. Grippo acknowledge partial support from ANPCyT PICT 2017-1315.}

\author[conicet,iam,ici]{Cristian M. Conde}
\ead{cconde@campus.ungs.edu.ar}

\author[conicet,ici]{Ezequiel Dratman}
\ead{edratman@campus.ungs.edu.ar}

\author[conicet,ici]{Luciano N. Grippo}
\ead{lgrippo@campus.ungs.edu.ar}

\address[conicet]{Consejo Nacional de Investigaciones
Cient{\' i}ficas y T{\' e}cnicas, Argentina}
\address[iam]{Instituto Argentino de Matem{\' a}tica Alberto Calder{\' o}n}
\address[ici]{Instituto de Ciencias, Universidad Nacional de General Sarmiento, Argentina}





\begin{abstract}
The edge-connectivity of a graph is the minimum number of edges whose deletion disconnects the graph. Let $\Delta(G)$ the maximum degree of a graph $G$ and let $\rho(G)$ be the spectral radius of $G$. In this article we present a lower bound for $\Delta(G)-\rho(G)$ in terms of the edge connectivity of $G$, where $G$ is a nonregular distance-hereditary graph. We also prove that $\rho(G)$ reaches the maximum at a unique graph in $\mathcal G$, when  $\vert V(G)\vert = n$, and $\mathcal G$ either is in the class of graphs with bounded tree-width or is in the class of block graphs with prescribed independence number. 
\end{abstract}


\end{frontmatter}


\section{Introduction}

To find lower and upper bounds for the spectral radius of a graph is a problem that have attracted the attention of many researchers. Probably, one of the most important motivations for studying this topic is due to a problem posted by Brualdi and Solheid in~\cite{BS-1986}. They proposed, in that article, to characterize the graphs having the maximum spectral radius among graphs on $n$ vertices and in a determined class of graphs. Since then, a wide variety of results on this topic have been published. In addition, finding bounds for the spectral radius of any graph in terms of nonspectral parameters is interesting enough. Many works can be found in the specialized literature. A recently published book summarizes most of the results related to this topic~\cite{Dragan2015}.

It is well known that the spectral radius of a graph $G$ is at most $\Delta(G)$ where $\Delta(G)$ stands for the maximum degree of $G$. In addition if $G$ is connected, the equality holds if and only if $G$ is a regular graph~\cite{Bapat}, meaning all of its vertices have the same degree. So, it is interesting to compare how far is the spectral radius $\rho(G)$ of a nonregular graph from $\Delta(G)$. In this direction, we can find in the literature the following two results among others. 

\begin{thm}~\cite{Cioba2007}
\label{thm: cioba}
If $G$ is a nonregular graph, then 

\[\Delta(G)-\rho(G)\ge \frac{(\vert V(G)\vert\Delta(G)-2\vert E(G)\vert)}{\vert V(G)\vert (D(\vert V(G)\vert\Delta(G)-2\vert E(G)\vert)+1)}\]
\end{thm}

\begin{thm}~\cite{Cioba2007bis}
\label{thm: cioba mejorado}
If $G$ is a nonregular graph, then 

\[\Delta(G)-\rho(G)\ge\frac{1}{D(G)|V(G)|}\]

\end{thm}

Notice that the lower bound obtained in Theorem~\ref{thm: cioba mejorado} improves that given in Theorem~\ref{thm: cioba}. In this paper we present a lower bound for those nonregular graphs within the class of distance-hereditary graphs that improves in some cases that of Theorem~\ref{thm: cioba mejorado} (see Theorem~\ref{thm: lower bound DH}).

In 2004, Hong published a result presenting an upper bound for the spectral radius of graph having tree-width $k$~\cite{Hong2004}, whose proof relies on an upper bound of Hong, Shu and Fang for the spectral radius of a graph on $n$ vertices in terms of the number of edges and the minimum degree~\cite{HSF2001}, where those graph satisfying the equality were also characterized. Since the proof of that result is involving, we decided to present a simpler one as an application of Lemma~\ref{lem: k-trees} whose demonstration only use rudiments of linear algebra.

Lu and Lin find the only graph which maximizes the spectral radius among trees with prescribed independence number~\cite{JiLu2016}. In~\cite{LuLin2015}, the authors find the unique connected graph on $n$ vertices, with given connectivity  and prescribed independence number having maximal spectral radius. Our contribution, in that line of work, is to find the unique block graph on $n$ vertices and given independence number with maximal spectral radius. Indeed, we have been able to prove that the pineapple on $n$ vertices having maximum independent set $\alpha$ is that unique graph. Notice also that block graphs is a superclass of trees. 

This article is organized as follows. In Section~\ref{sec: preliminaries}, we introduce some preliminary results and definitions. In Section~\ref{sec: DH graphs} we present a lower bound for $\Delta(G)-\rho(G)$, when $G$ is a distance-hereditary graph. Section~\ref{sec: k-trees} is devoted to present a simpler proof of Theorem 2.1, which appears in~\cite{Hong2004}, related to the maximum spectral radius among all $k$-trees. Finally, in Section~\ref{sec: block graphs} we find the unique block graph on $n$ vertices and given independence number with maximum spectral radius.

\section{Preliminaries}\label{sec: preliminaries}

\subsection{Definitions}

All graphs, mentioned in this article, are finite, have no loops and multiple edges. Let $G$ be a graph. We use $V(G)$ and $E(G)$ to denote the set of vertices and the set of edges of $G$, respectively. We denote by $|X|$ the cardinality of a finite set $X$. Let $v$ be a vertex of $G$, $N_G(v)$ (resp. $N_G[v]$) stands for the neighborhood of $v$ (resp. $N_G(v)\cup\{v\}$), if the context is clear the subscript $G$ will be omitted. We use $d_G(v)$ to denote the degree of $v$ in $G$, or $d(v)$ provided the context is clear. A vertex of degree $|V(G)|-1$ is called \emph{universal vertex}. By $\overline G$ we denote the complement graph of $G$. Given a set $F$ of edges of $G$ (resp. of $\overline G$), we denote by $G-F$ (resp. $G+F$) the graph obtained from $G$ by removing (resp. adding) all the edges in $F$. If $F=\{e\}$ we use $G-e$  (resp. $G+e$) for short. Let $X\subseteq V(G)$, we use $G[X]$ to denote the graph induced by $X$. By $G-X$ we denote the graph $G[V(G)\setminus X]$. If $X=\{v\}$ we use $G-v$ for short. Given $u,v\in V(G)$, a path $P$ of $G$ is said to be an \emph{$u,v$-path} if $u$ and $v$ are the endpoints of $P$. The \emph{distance} between $u$ and $v$ is the minimum number of edges of an $u,v$-path. The \emph{diameter} of $G$, denoted $D(G)$, is the maximum distance among all pair of vertices of $G$. A set of edges $S$, possibly empty, such that $G-S$ has more than one connected component is said to be a~\emph{disconnecting set}. The~\emph{edge-connectivity} of $G$, denoted $\kappa'(G)$, is the minimum size of a disconnecting set. We said that a set of edges $F$, possibly empty, is a \emph{$u,v$-disconnecting set} if $u$ is in a different connected component of $G-F$ from that in which is $v$. Let $A, B\subseteq V(G)$ we said that $A$ is \emph{complete to} (resp. \emph{anticomplete to}) $B$ if every vertex in $A$ is adjacent (resp. nonadjacent) to every vertex of $B$. We denote by $\kappa'(u,v)$ to the minimum size of a $u,v$-disconnecting set, and by $\lambda'(u,v)$ we denote to the maximum number of edge-disjoint $u,v$-paths. Notice that $\kappa'(G)$ is the minimum $\kappa'(u,v)$ among all pairs of vertices $u$ and $v$ of $G$. Clearly, $\lambda'(u,v)\ge\kappa'(u,v)$. Besides, it is well-known that $\lambda'(u,v)=\kappa'(u,v)$ (see for instance~\cite{West2000}). A set of pairwise nonadjacent vertices of $G$ is called an \emph{independent set} (or \emph{stable set}). The \emph{independence number} of $G$, denoted $\alpha(G)$, is the maximum cardinality of an independence number of $G$. A \emph{clique} is a set of pairwise adjacent vertices.

A \emph{complete graph} on $n$ vertices, denoted $K_n$, is a graph consisting of $n$ pairwise adjacent. A \emph{tree} is a connected and acyclic graph. By $K_{1,n-1}$ we denote the tree on $n$ vertices having a universal vertex. A \emph{leaf} of a tree is a vertex of degree one and a \emph{support vertex} in a tree is the only vertex adjacent to a leaf. Given two graphs $G$ and $H$, we use $G=H$ to denote that $G$ and $H$ are isomorphic graphs. A $k$-tree is defined inductively as follows: $K_k$ is a $k$-tree, adding a vertex to a $k$-tree, adjacent to a clique on $k$ vertices, is also a $k$-tree. A graph $G$ is \emph{distance-hereditary} if for every connected induced subgraph $H$ of $G$ and every pair of vertices in $H$ the distance between them in $H$ is the same as the distance in $G$. For more details about this graph class the reader is referred to~\cite{Oum2005} and all references therein. The \emph{tree-width} of a graph is the minimum $k$ for which there exists a $k$-tree $T_k$ such that $G$ is a subgraph of $T_k$. Notice that the tree-width of a tree is equal to one. This parameter is relevant from an algorithmic point of view as well as structural. There are other ways to define the tree-width of a graph that can be found in~\cite{BeinekeWilson}.

Let $G$ be a graph. We denote by $A(G)$ the adjacency matrix of $G$, and $\rho(G)$ stands for the spectral radius of $A(G)$, we refer to $\rho(G)$ as the spectral radius of $G$. Perron-Frobenius theorem implies that the principal eigenvector of $A(G)$ has all its entries either positive or negative. In addition, $\rho(G)$ coincides with the maximum eigenvalue of $G$. The reader is referred to~\cite[Ch. 6]{Bapat} for a simple proof of this observation. If $x$ is the principal eigenvector of $A(G)$ which is clearly indexed by $V(G)$, we use $x_u$ to denote the coordinate of $x$ corresponding to the vertex $u$.

\subsection{Some results}

Adding edges to a graph increases the spectral radius of a graph.

\begin{lem}\label{lem: adding edges}
If $G$ is a graph such that $uv\notin E(G)$, then $\rho(G)<\rho(G+uv)$. 
\end{lem}

Some results, in connections with finding those graphs that maximizes the spectral radius of a graph on $n$ vertices within a given class $\mathcal H$ of graphs, have been solved by means of graphs transformations that increases the spectral radius.  We refer to the reader to~\cite{Dragan2015} for more details about this and other techniques. Notice that if $\mathcal H$ contains the complete graphs, then $K_n$ maximizes $\rho(G)$ for every $G\in\mathcal H$, because of Lemma~\ref{lem: adding edges}. Lov\' asz and Pelik\'an in~\cite{LP1973} prove that the unique graph with maximum spectral radius among the trees on $n$ vertices is the star $K_{1,n-1}$ defining a partial order within the trees by means of their characteristic polinomials.

\begin{thm}\cite{LP1973}\label{thm: maximum-trees}
If $T$ is a tree on $n$ vertices, then $\rho(T)\le \sqrt{n-1}$. In addition, the equality holds if and only if $T=K_{1,n-1}$. 
\end{thm}

Nevertheless, in order to easily prove this result, using the technique of graph transformations, the following result can be used. 

\begin{lem}\label{lem: moving-vertices}\cite{Rowlinson1990}
Let $G$ be a connected graph and let $u$ and $v$ two vertices of $G$ such that $x_u\le x_v$. If $\{v_1,\ldots,v_r\}\subseteq N(u)\setminus N(v)$, then 

\[\rho(G)<\rho(G-\{uv_1,\ldots,uv_r\}+\{vv_1,\ldots,vv_r\}).\]
\end{lem}

Lemma~\ref{lem: moving-vertices} was proved by the first time in~\cite{Rowlinson1990} but for an easy proof the reader is referred to~\cite{Wu2005}. We would like to point out that Theorem~\ref{thm: maximum-trees} can be proved, using Lemma~\ref{lem: moving-vertices} by showing that if $T$ is a tree on $n$ vertices having the maximum spectral radius then there is only one support vertex. Otherwise there would exist two support vertices $u$ and $v$ in $T$ satisfying $x_u\le x_v$ and thus if $w$ is a leaf adjacent to $u$, and nonadjacent to $w$,  then $\rho(T)<\rho(T-uw+vw)$. Therefore, the tree having the maximum spectral radius is $K_{1,n-1}$ whose only support vertex is its vertex of degree $n-1$. Lemma~\ref{lem: moving-vertices} can be generalized and this generalization turns out to be helpful to deal with $k$-trees and block graphs as we will see in sections~\ref{sec: k-trees} and~\ref{sec: block graphs}. 

In the following lemma  we consider a set of vertices $u_1,\ldots,u_{\ell}$ of a graph $G$, where $x_i$ stands for $x_{u_i}$ for every $1\le i\le\ell$.

\begin{lem}\label{lem: k-trees}
Let $G$ be a connected graph and let $u_1,\ldots,u_k,u_{k+1},\ldots,u_{\ell}$ be vertices of $G$ such that
$\sum_{i=1}^k x_i\le \sum_{i=k+1}^{\ell} x_i$, and let $W\subseteq V(G)\setminus \{u_1,\ldots,u_{\ell}\}$. If $\{u_1,\ldots,u_k\}$ is complete to $W$ and $\{u_{k+1},\ldots,u_{\ell}\}$ is anticomplete to $W$, then 
\[\rho(G)<\rho(G-\{wu_i: w\in W\mbox{ and }1\le i\le k\}+\{wu_i: w\in W\mbox{ and }k+1\le i\le\ell\}).\]
\end{lem}

\begin{proof}
Let $x$ be the principal eigenvector of $G$ such that $x_i>0$  for every $1\le i\le |V(G)|$ and $\|x\|=1$ the existence of such principal eigenvector is guaranteed by Perron-Frobenius theorem. 

Let $G^*=G-\{wu_i: w\in W\mbox{ and }1\le i\le k\}+\{wu_i: w\in W\mbox{ and }k+1\le i\le\ell\}$. Hence,
\begin{equation}
x^t(A(G^*)-A(G))x  = 2\sum_{w\in W}x_w \left(\sum_{i=k+1}^{\ell}x_i-\sum_{i=1}^k x_i\right),
\end{equation} 
Since $\sum_{i=1}^k x_i\le \sum_{i=k+1}^{\ell} x_i$ and $x_w>0$ for every $w\in W$, $x^t(A(G^*)-A(G))x\ge 0$ and thus 

\begin{equation}\label{eq: spectral radius}
\begin{split}
\rho(G^*)&=\max_{y:\|y\|=1} y^t A(G^*)y\\
&\ge x^t A(G^*)x\\
&\ge x^t A(G)x =\rho(G).
\end{split}
\end{equation}

Suppose, towards a contradiction, that $\rho(G^*)=\rho(G)$. By Inequality~\eqref{eq: spectral radius}, $A(G^*)x=\rho(G^*)x$. Thus, on the one hand, 

\begin{equation}\label{eq: rho star}
\rho(G)x_1 =\sum_{uu_1\in E(G^*)}x_u+\sum_{w\in W}x_w,
\end{equation}

and on the other hand

\begin{equation}\label{eq: rho}
\rho(G^*)x_1 =\sum_{uu_1\in E(G^*)}x_u.
\end{equation}

Since $x_1>0$ and $x_w>0$ for all $w\in W$, equations~\eqref{eq: rho star} and~\eqref{eq: rho} implies $\rho(G^*)<\rho(G)$, a contradiction. The contradiction arose from supposing that $\rho(G)=\rho(G^*)$. Therefore, $\rho(G)<\rho(G^*)$.
\end{proof}

It is worth mentioning that Lemma~\ref{lem: k-trees} was presented in~\cite{LuLin2015} by Lu and Lin but in an slightly different way. They prove that $\rho(G)\le\rho(G^*)$ when $\vert W\vert=1$ and that the inequality is strict when  $\sum_{i=1}^k x_i< \sum_{i=k+1}^{\ell} x_i$. So, we decided to write the proof  for the sake of completion.

\section{Nonregular distance-hereditary graphs}\label{sec: DH graphs}

We will proceed to show a lower bound for $\Delta(G)-\rho(G)$ when $G$ is a connected nonregular distance-hereditary graph.

\begin{thm}\label{thm: lower bound DH}
Let $G$ be a connected nonregular distance-hereditary graph of diameter $D$, then

\[\Delta(G)-\rho(G)\ge \frac{(\vert V(G)\vert\Delta(G)-2\vert E(G)\vert)\kappa'(G)}{\vert V(G)\vert (D(G)(\vert V(G)\vert\Delta(G)-2\vert E(G)\vert)+\kappa'(G))}.\]	
\end{thm}

\begin{proof}
Let $x$ be the principal eigenvector of $A(G)$; i.e. $A(G)x=\rho(G)x$. In addition $x$ is chosen with all its entries positive and $\|x\|=1$. Thus

\begin{equation}\label{eq: igualdad base}
\Delta(G)-\rho(G)=\sum_{u\in V(G)}(\Delta(G)-d_G(u))x_u^2+\sum_{uv\in E(G)}(x_u-x_v)^2
\end{equation}

Details on Eq.~\eqref{eq: igualdad base} can be found in~\cite[Page 53]{Dragan2015}. Let $a$ be a vertex of $G$ corresponding to the minimum principal eigenvalue component, denoted $\eigenmin$, and let $b$ be a vertex of $G$ corresponding to a maximum principal eigenvector component, denoted $\eigenmax$.

Since $\lambda'(a,b)\ge\kappa'(G)$, there exist $\ell$ edge-disjoint $a,b$-paths $P_1,\ldots,P_{\ell}$ with $\ell\ge \kappa'(G)$ (see~\cite[Page 168]{West2000}). We use this fact to give a lower bound to the second term of Eq.~\eqref{eq: igualdad base}. Notice also that if $D=D(G)$, since $G$ is a distance-hereditary graph, then $\vert E(P_i)\vert\le D$ for every $1\le i\le\ell$. Hence,

\begin{equation}\label{eq: desigualdad I}
\begin{split}
\sum_{uv\in E(G)}(x_u-x_v)^2&\ge \sum_{i=1}^\ell\left(\sum_{uv\in E(P_i)} (x_u-x_v)^2\right)\\
&\ge\sum_{i=1}^\ell\frac{1}{\vert E(P_i)\vert}\left(\sum_{uv\in E(P_i)}(x_u-x_v)\right)^2\\
&\ge\frac{\ell}{D}(\eigenmax-\eigenmin)^2\\
&\ge \frac{\kappa'(G)}{D}(\eigenmax-\eigenmin)^2.
\end{split}
\end{equation}

The second relation is a consequence of Cauchy-Schwarz inequality. Ciob\u a et al. notice that, since $G$ is nonregular and thus $\Delta(G)-d(u)>0$ for at least one vertex $u\in V(G)$. The following inequality, proved in~\cite{Cioba2007},  holds
\begin{equation}\label{eq: desigualdad II}
\sum_{u\in V(G)}(\Delta(G)-d(u))x_u^2\ge (\vert V(G)\vert \Delta(G)-2\vert E(G)\vert)\eigenmin^2.
\end{equation} 
Combining~\eqref{eq: igualdad base},~\eqref{eq: desigualdad I} and~\eqref{eq: desigualdad II} it follows
\begin{equation}\label{eq: desigualdad III}
\Delta(G)-\rho(G)\ge (\vert V(G)\vert \Delta(G)-2\vert E(G)\vert)\eigenmin^2 + \frac{\kappa'(G)}{D}(\eigenmax-\eigenmin)^2.
\end{equation}
If we consider the quadratic function 
\[f(x)=(\vert V(G)\vert \Delta(G)-2\vert E(G)\vert)x^2 + \frac{\kappa'(G)}{D}(\eigenmax-x)^2,\]

we can see that its minimum is reached at $\widehat{x}=\frac{\kappa'(G)\eigenmax}{D(\vert V(G)\vert \Delta(G)-2\vert E(G)\vert)+\kappa'(G)}$. Since $f(\widehat{x})=\frac{\kappa'(G)}{D(\vert V(G)\vert \Delta(G)-2\vert E(G)\vert)+\kappa'(G)}$, it follows from inequality~\eqref{eq: desigualdad III} that

\[\Delta(G)-\rho(G)\ge\frac{\kappa'(G)(\vert V(G)\vert\Delta(G)-2\vert E(G)\vert)}{D(\vert V(G)\vert \Delta(G)-2\vert E(G)\vert)+\kappa'(G)}.\]

\end{proof}

\begin{rmk}
Notice that our lower bound improves the lower bound of Theorem~\ref{thm: cioba mejorado} whenever $D(\vert V(G)\vert \Delta(G)-2\vert E(G)\vert)>2$, $\kappa'(G)>2$, and $G$ is a connected nonregular distance-hereditary graph.  
\end{rmk}

\section{Graphs with bounded tree-width}\label{sec: k-trees}

A \emph{simplicial} vertex of a graph $G$ is a vertex $v$ such that $N(v)$ is a clique. Notice that if $G$ is a $k$-tree and $v$ is a simplicial vertex of $G$ then $|N[v]|=k+1$. We use $S_{k,n-k}$ to denote the graph on $n$ vertices whose vertex set can be partitioned into a clique $Q$ on $k$ vertices and an independent set $I$ on $n-k$ vertices, and $I$ is complete to $Q$. Notice that $S_{k,n-k}$ is a $k$-tree and every $k$-tree distinct of a complete graph  has at least two nonadjacent simplicial vertices.


\begin{thm}\label{thm: tree-width}~\cite{Hong2004}
If $G$ is a graph on $n$ vertices with tree-width equals $k$, then
\[\rho(G)\le\frac{k-1+\sqrt{4kn-(k+1)(3k-1)}}{2}.\]
In addition, the equality holds if and only if $G=S_{k,n-k}$. 
\end{thm}

\begin{proof}
Suppose that $G$ is a $k$-tree on $n$ vertices, having maximum spectral radius among all $k$-trees on $n$ vertices. Suppose, towards a contradiction, that  $G$ has two simplicial vertices $u$ and $v$ such that $\vert N(u)\cap N(v)\vert <k$ (see Fig.~\ref{fig: ktree}). Assume, without losing generality, that $\sum_{w\in N(u)\setminus N(v)}x_w\le\sum_{w\in N(v)\setminus N(u)}x_w$. Therefore, by Lemma~\ref{lem: k-trees}, $\rho(G)<\rho(G^*)$, where $G^*=G-\{wu\in E(G):\mbox{ }w\in N(u)\setminus N(v)\}+\{wv\in E(G):\mbox{ }w\in N(u)\setminus N(v)\}$, since $G^*$ is a $k$-tree, this contradicts that $G$ is the $k$-tree having maximum spectral radius. The contradiction arose from supposing that $u$ and $v$ are two simplicial vertices of $G$ such that $\vert N(u)\cap N(v)\vert <k$. Hence, since $G$ is a $k$-tree, $|N(u)\cap N(v)|=k$.  Consequently, every vertex $x\in V(G)\setminus N(u)$ is complete to $N(u)$. Therefore, $G=S_{k,n-k}$.  We have already proved that $S_{k,n-k}$ is the unique graph that maximizes the spectral radius among all $k$-trees.  

\begin{figure}
\centering
    \includegraphics[scale=0.6]{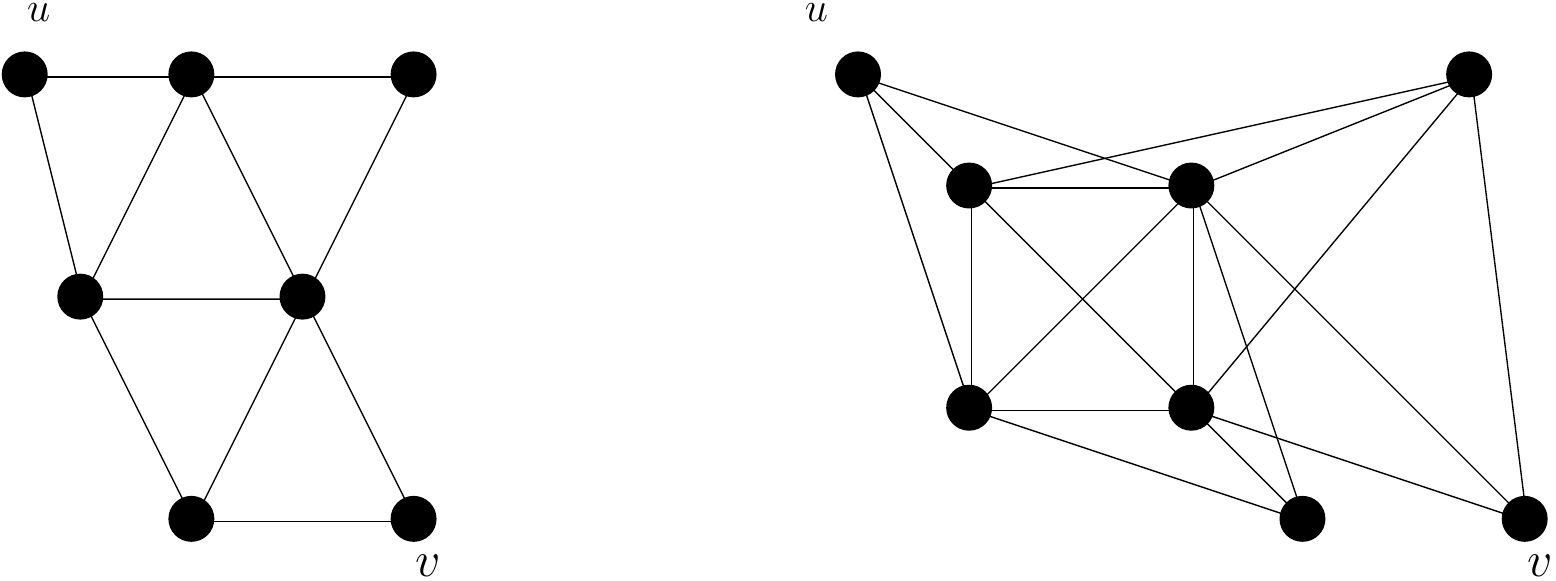}   
    \caption{The graph on the left is a $2$-tree and the graph on the right is a $3$-tree. Both graph marked vertices $u$ and $v$ satisfying $|N(u)\cap N(v)|<k$ for every $k\in\{2,3\}$.}
    \label{fig: ktree}
\end{figure}

Consider any graph $G$ having tree-width equal to $k$. Hence there exists a $k$-tree $H$ for which $G$ is a subgraph of it. By Lemma~\ref{lem: adding edges} and the conclusion of the above paragraph $\rho(G)\le\rho(H)\le \rho(S_{k,n-k})$. To finish the proof we only need to compute the spectral radius of $S_{k,n-k}$. By symmetry, assume that all of the coordinates of the  principal eigenvector corresponding to the universal vertices of $S_{k,n-k}$ are equal to $x$ and all of the coordinates of the principal eigenvector  corresponding to the vertices of degree $k$  are equal to $y$.  Hence $kx=\lambda y$ and $(k-1)x+(n-k)y=\lambda x$, where $\lambda=\rho(S_{k,n-k})$.  If $n\neq k$, then $x(\lambda^2-(k-1)\lambda-(n-k)k)=0$. Since $x>0$, 

\[\lambda=\frac{k-1+\sqrt{4kn-(k+1)(3k-1)}}{2}.\]

Notice that this formula holds even when $n=k$.
\end{proof}

\begin{rmk}
Theorem~\ref{thm: tree-width} generalizes Theorem~\ref{thm: maximum-trees}.
\end{rmk}

\section{Block graphs with given independence number}\label{sec: block graphs}

The adjacecy matrix of block graph were studied by Bapat and Souvik in~ \cite{brs-2014}.Throughout of this section we will need some definitions and concepts introduced next. A vertex $v$ of a graph $G$ is a \emph{cut vertex} if $G-v$ has a number of connected components greater than the number of connected components of $G$. We use $\Si_1(G)$ to denote the set of simplicial vertices of $G$. 

Let $H$ be a graph. A \emph{block} of $H$ is a maximal subgraph $B$ of $H$ having no cut vertex in $B$. Blocks are also known as 2-connected components. A \emph{block graph} is a connected graph whose blocks are complete graphs. Notice that trees are block graphs. A \emph{simplicial block} of $H$ is a block $B$ having at least a simplicial vertex in $H$. Let $G$ be a block graph,  a \emph{leaf block} is a block of $G$ such that contains exactly one cut vertex of $G$. Notice that every vertex but one, in a leaf block, is a simplicial vertex. Notice that a noncomplete block graph is distance-hereditary.

\begin{figure}
\centering
    \includegraphics[scale=0.6]{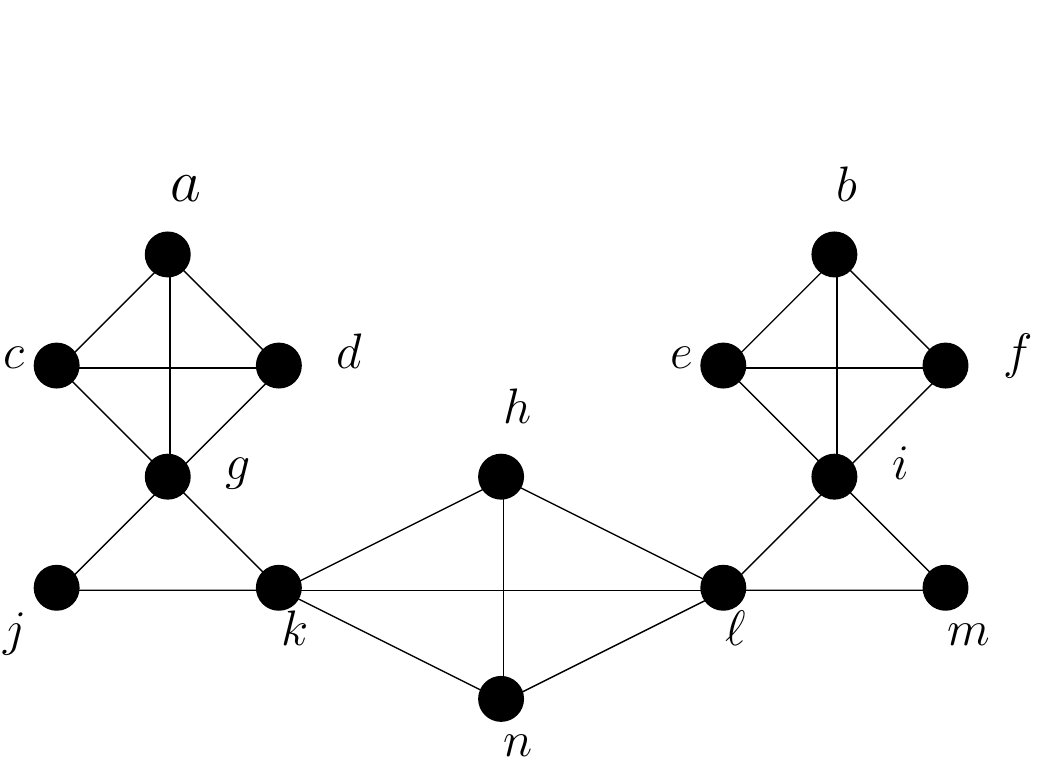}   
    \caption{A block graph, each of its blocks is a simplicial one and its has two leaf blocks.}
    \label{fig: block1}
\end{figure}

Next we will state and prove technical results needed to demonstrate the main statement of this section.

\begin{lem}\label{lem: independent set simpliciales}
Let $G$ be a block graph and let $B$ a simplicial block of $G$. If $S$ is a maximum independent set of $G$, then $\vert S\cap V(B)\vert=1$. In addition, such a maximum independent set $S$ can be chosen so that $S\cap V(B)=\{v\}$, where $v$ is any simplicial vertex of $G$ such that $N_G[v]=V(B)$.
\end{lem}

\begin{proof}
Consider a maximum independent set $S$ and let $B$ be a simplicial block of $G$. Hence there exists a vertex  $v\in\Si_1(G)$ such that $V(B)=N_G[v]$. Thus there exists a vertex $w\in V(B)\cap S$, because of the maximality of $S$, since otherwise $S\cup\{v\}$ would be an independent set. Therefore, $\vert V(B)\cap S\vert=1$. In addition, since $S$ is a maximum independent set, $S'=S\setminus\{w\}\cup\{v\}$ is also a maximum independent set. We can proceed in this way with each simplicial block in order to obtain a maximum independent set as stated in the lemma. 
\end{proof}

\begin{cor}\label{cor: maximum independence number leaf block}
Let $G$ be a block graph. Then, there exists a maximum independent set $S$ such that $S\cap V(B)=\{v\}$ for each leaf block $B$ of $G$, where $N_G[v]=V(B)$.
\end{cor}

\begin{proof}
It suffices to notice that if $B$ is a leaf block of $G$, then $B$ is a simplicial block of $G$. Therefore, the result immediately follows from Lemma~\ref{lem: independent set simpliciales}.
\end{proof}



Let $G$ be a block graph. We use $\Le(G)$ to denote the set of vertices of $G$ belonging to any leaf block. Let $B$ be a block of $G$. We use $L(B)$ to denote the set of simplicial vertices of those leaf blocks of $G$ having exactly one vertex in common with $V(B)$. By $\ell_G(B)$ we denote the number of these leaf blocks. When the context is clear enough we use $\ell(B)$ for short. In the graph depicted in Fig.~\ref{fig: block1}, if $B$ is the block induced by $\{g,j,k\}$, then $L(B)=\{a,c,d\}$ and $\ell(B)=1$; and in the graph depicted in Fig.~\ref{fig: block2}, if $B$ is the block induced by $\{g,j,k\}$, then $L(B)=\{a,c,d,n\}$ and $\ell(B)=2$.

\begin{cor}\label{cor: leaf block formula}
If $G$ is a block graph and  $B$ is a leaf block of $G$, then 
\[\alpha(G)=\alpha(G-B)+1.\] 
\end{cor}

\begin{proof}
The proof follows from Lemma~\ref{lem: independent set simpliciales}.
\end{proof}

\begin{lem}\label{lem: indpendent set}
Let $G$ be a block graph and let $B$ be a leaf block of $H=G-(\Le(G)\cap \Si_1(G))$, where $v$ is its only cut vertex of $H$ in $V(B)$. Then, the following conditions hold:
\begin{enumerate}
\item $\alpha(G)=\alpha(G-(V(B)\cup L(B)))+\ell(B)+1$, if $V(B)$ has a simplicial vertex in $G$.
\item $\alpha(G)=\alpha(G-((V(B)\cup L(B))\setminus\{v\}))+\ell(B)$, if $V(B)$ has no simplicial vertex in $G$.
\end{enumerate}
\end{lem}

\begin{proof}
Let $G$ be a block graph and let $v$  the only cut vertex of $H$ in $V(B)$. Notice that $H$ is the graph obtained from $G$ by removing every simplicial vertex belonging to a leaf block of $G$. By Corollary~\ref{cor: maximum independence number leaf block}, $G$ has a maximum independent set $S$ such that if $B'$ is any leaf block of $G$ having a simplicial vertex $w$, then $V(B')\cap S=\{w\}$. Hence if $V(B)$ has a simplicial vertex in $G$, then $v\notin S$. Therefore, the only vertices of $V(B)\cup L(B)$ in $S$ are those simplicial vertices in a leaf block of $G$ having a vertex in common with $V(B)$ and exactly one of the simplicial vertices of $G$ in $V(B)$, the remaining vertices of $S$ are in $V(G)\setminus(V(B)\cup L(B))$ and the result holds. If $V(B)$ does not have any simplicial vertex of $G$, then $(V(B)\cup L(B))\setminus\{v\}$ has in $S$ exactly one vertex for each leaf block of $G$ having a vertex in common with $V(B)$ and $v$ might belong or not to $S$, thus the second statement holds. 
\end{proof}

\begin{figure}
\centering
    \includegraphics[scale=0.6]{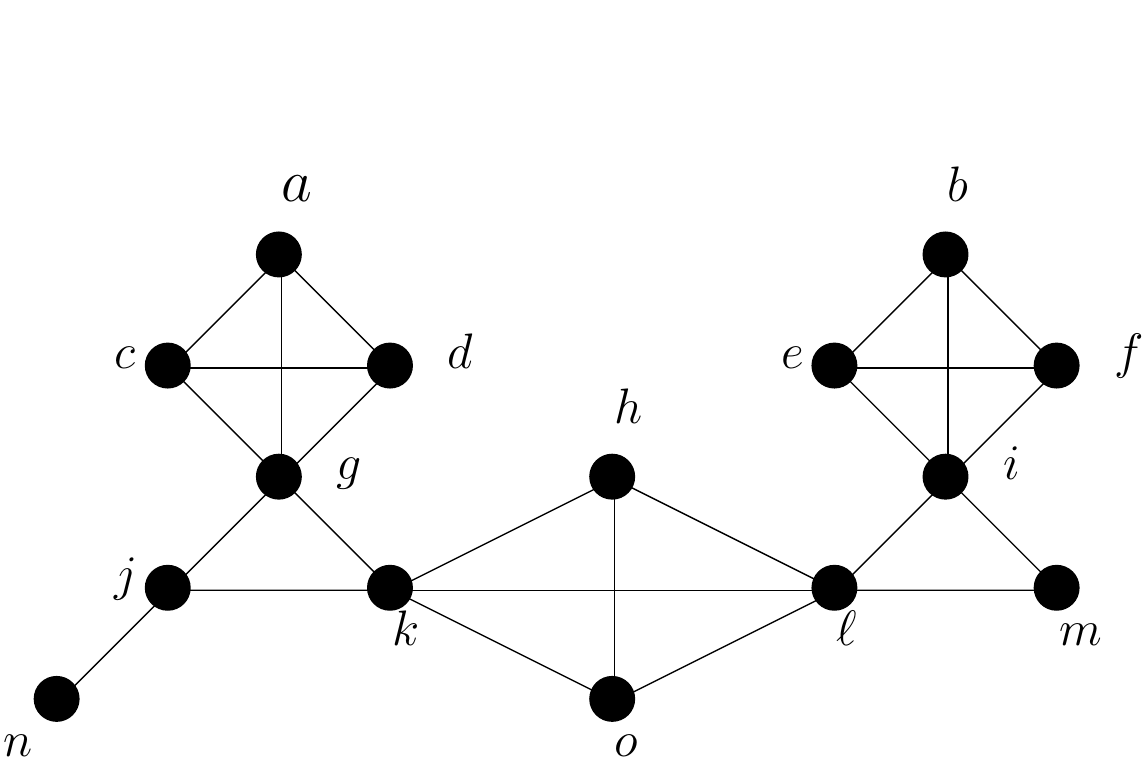}   
    \caption{Block graphs with three leaf blocks, two nonleaf simplicial blocks and a nonsimplicial block.}
    \label{fig: block2}
\end{figure}

The following lemma will allow to describe with more precision the structure of those block graphs with prescribed independence number having maximum spectral radius. Recall that two blocks in a graph have at most one vertex in common, which is also a cut vertex.

\begin{lem}\label{lem: bloques separados}
If $G$ is a block graph with maximum spectral radius among all block graphs with independence number $\alpha$, and $B_1$ and $B_2$ are leaf blocks of $G-(\Le(G)\cap\Si_1(G))$, then $\vert V(B_1)\cap V(B_2)\vert =1$.
\end{lem}

\begin{proof}
Suppose, towards a contradiction, that $V(B_1)\cap V(B_2)=\emptyset$. Assume that $v_i$ is the only cut vertex in $V(B_i)$ that does not belong to a leaf block of $G$, for each $i\in\{1,2\}$. We are going to split the proof into three cases. Let $x$ be a principal eigenvector of $G$ having all its coordinates positive.

\vspace{0.5cm}
\emph{Case 1: $V(B_i)\cap\Si_1(G)\neq\emptyset$ for each $i\in\{1,2\}$}.
\vspace{0.5cm}

Let $S_i=V(B_i)\cap\Si_1(G)$ for each $i\in\{1,2\}$. Consider for instance the graph depicted in Fig.~\ref{fig: block1} where those blocks playing the roles of $B_1$ and $B_2$ are those induced by $\{g,j,k\}$  and  $\{i,\ell,m\}$, respectively. In this case $v_1=k$, $v_2=\ell$,  $S_1=\{j\}$ and $S_2=\{m\}$.

By Lemma~\ref{lem: indpendent set} we know that

\[\alpha=\alpha\left(G-\left(\bigcup_{i=1}^2(V(B_i)\cup L(B_i))\right)\right)+\ell_G(B_1)+\ell_G(B_2)+2.\]

Assume, without losing generality, that $\sum_{a\in S_1} x_a+x_{v_1}\le \sum_{a\in S_2}x_a+x_{v_2}$. We construct a graph $G^*$ from $G$ as follows. We delete every edge $sv$ with $s\in S_1\cup\{v_1\}$ and $v\in V(B_1)\setminus( S_1\cup \{v_1\})$  and then we add every edge  $vw$ with $v\in V(B_1)\setminus( S_1\cup \{v_1\})$ and $w\in V(B_2)$. Clearly, $G^*$ is a block graph and its block $B'$ whose vertex set is $(V(B_1)\setminus( S_1\cup \{v_1\}))\cup V(B_2)$ has at least a simplicial vertex because $V(B_2)$ has a simplicial vertex in $G$, and the block $B''$ induced by $S_1\cup\{v_1\}$ in $G^*$ is a leaf block of $G^*$. Besides, $B'$ is a leaf block of $G^*-(\Le(G^*)\cap \Si_1(G^*))$ having a simplicial vertex and $\ell_{G^*}(B')=\ell_G(B_1)+\ell_G(B_2)$. By Lemmas~\ref{lem: adding edges} and~\ref{lem: k-trees}, $\rho(G)<\rho(G^*)$. In virtue of Lemma~\ref{lem: indpendent set} and Corollary~\ref{cor: leaf block formula} applied to $B''$
\begin{align*}
\alpha(G^*) & =\alpha(G^*-V(B''))+1\\
			&=\alpha((G^*-V(B''))\setminus(V(B')\cup L(B')))+\ell_{G^*}(B')+2\\
						& =  \alpha\left(G-\left(\bigcup_{i=1}^2(V(B_i)\cup L(B_i))\right)\right)+\ell_G(B_1)+\ell_G(B_2) +2.
\end{align*}
We reach a contradiction.

\vspace{0.5cm}
\emph{Case 2:  Exactly one of $V(B_1)$ or $V(B_2)$ has a simplicial vertex of $G$.}
\vspace{0.5cm}

Assume, without losing generality, that $V(B_1)$ has at least one simplicial vertex of $G$ and $V(B_2)\cap\Si_1(G)=\emptyset$. Let $S=V(B_1)\cap\Si_1(G)$

Consider for instance the graph depicted in Fig.~\ref{fig: block2} where those blocks playing the roles of $B_1$ and $B_2$ are those induced by $\{i,\ell,m\}$  and $\{g,j,k\}$, respectively. In this case $v_1=\ell$, $v_2=k$ and $S=\{m\}$.

By Lemma~\ref{lem: indpendent set} we know that
\[\alpha=\alpha\left(G-\left(\left(\bigcup_{i=1}^2(V(B_i)\cup L(B_i))\right)\setminus\{v_2\}\right)\right)+\ell_G(B_1)+\ell_G(B_2)+1.\]

If $x_{v_2}\le \sum_{a\in S}x_a+x_{v_1}$, then the block graph $G^*$ obtained by deleting every edge $bv_2$ with $b\in (V(B_2)\setminus\{v_2\})$ and by adding every edge $bv_1$ with $b\in (V(B_2)\setminus\{v_2\})$ satisfies, by Lemma~\ref{lem: moving-vertices}, that $\rho(G)<\rho(G^*)$.  Notice that $B_1$ is a block of $G^*$ having at least one simplicial vertex such that $\ell_{G^*}(B_1)=\ell_{G}(B_1)$, and $B'_2=G^*[(B_2-v_2)\cup\{v_1\}]$ is a block of $G^*$ having no simplicial vertices such that $\ell_{G^*}(B'_2)=\ell_G(B_2)$. Besides, both of $B'_1$ and $B'_2$ are leaf blocks of $G^*\setminus(\Le(G^*)\cap\Si_1(G^*))$, where $B'_1=B_1$, and thus by Lemma~\ref{lem: indpendent set}
\begin{align*}
\alpha(G^*)&=\alpha\left( G^*-\left(\bigcup_{i=1}^2(V(B'_i)\cup L(B'_i))\right)\right)+\ell_{G^*}(B'_1)+\ell_{G^*}(B'_2)+1\\
					 &=\alpha\left(G-\left(\left(\bigcup_{i=1}^2(V(B_i)\cup L(B_i))\right)\setminus\{v_2\}\right)\right)+\ell_{G}(B_1)+\ell_{G}(B_2)+1.
\end{align*}
Thus we reach a contradiction.

Suppose now that $x_{v_2}\ge\sum_{a\in S}x_a+x_{v_1}$.  We construct a block graph $G^*$ from $G$ as follows. We delete every edge $sv$ with $s\in S\cup\{v_1\}$ and $v\in V(B_1)\setminus (S\cup\{v_1\})$,  and we add every edge $vw$ with $v\in V(B_1)\setminus (S\cup\{v_1\})$  and $w\in V(B_2)$. Clearly, the block $B'$ of $G^*$ whose vertex set is $V(B_1- (S\cup\{v_1\}))\cup V(B_2)$ has no simplicial vertex of $G^*$ and $\ell_{G^*}(B')=\ell_G(B_1)+\ell_G(B_2)$, and the block $B''$ induced in $G^*$ by $S\cup\{v_1\}$  is a leaf block. By Lemmas~\ref{lem: adding edges} and~\ref{lem: k-trees}, $\rho(G)<\rho(G^*)$. By Lemma~\ref{lem: indpendent set} and Corollary~\ref{cor: leaf block formula} 
\begin{align*}
\alpha(G^*)&=\alpha((G^*-V(B''))\setminus((V(B')\cup L(B'))\setminus\{v_2\}))+\ell_{G^*}(B')+1\\
					 &=\alpha\left(G-\left(\left(\bigcup_{i=1}^2(V(B_i)\cup L(B_i))\right)\setminus\{v_2\}\right)\right)+\ell_G(B_1)+\ell_G(B_2)+1.	
\end{align*}

We reach a contradiction.

\vspace{0.5cm}

\emph{Case 3:  $V(B_i)$ has no simplicial vertex of $G$ for each $i\in\{1,2\}$.} 

\vspace{0.5cm}

Consider for instance the graph depicted in Fig.~\ref{fig: block3} where those blocks playing the roles of $B_1$ and $B_2$ are those induced by $\{g,j,k\}$ and $\{i,\ell,m\}$, respectively. In this case $v_1=k$ and $v_2=\ell$.

By Lemma~\ref{lem: indpendent set} we know that
\[\alpha=\alpha(G)=\alpha\left(G-\left(\left(\bigcup_{i=1}^2(V(B_i)\cup L(B_i))\right)\setminus\{v_1,v_2\}\right)\right)+\ell_G(B_1)+\ell_G(B_2).\]
Assume, without losing generality, that $x_{v_1}\ge x_{v_2}$. We transform $G$ into the block graph $G^*$ by deleting every edge $v_2u$ with $u\in V(B_2)\setminus\{v_2\}$ and adding every edge $v_1u$ with $u\in V(B_2)\setminus\{v_2\}$. By Lemma~\ref{lem: moving-vertices}, $\rho(G)<\rho(G^*)$. Let define the blocks $B'_1$ and $B'_2$ of $G^*$ as those induced by $V(B_1)$ and $V((B_2-v_2)\cup\{v_1\})$, respectively. In addition, $B'_1$ and $B'_2$ are blocks of $G^*-(\Le(G^*)\cap \Si_1(G^*))$ such that $\ell_{G^*}(B'_i)=\ell_G(B_i)$ for each $i\in\{1,2\}$. Besides, by Lemma~\ref{lem: indpendent set}, 
\begin{align*}
\alpha(G^*)	&=\alpha\left(G^*-\left(\left(\bigcup_{i=1}^2(V(B'_i)\cup L(B'_i))\right)\setminus\{v_1\}\right)\right)+\ell_{G^*}(B'_1)+\ell_{G^*}(B'_2)\\
						&=\alpha\left(G-\left(\left(\bigcup_{i=1}^2(V(B_i)\cup L(B_i))\right)\setminus\{v_1,v_2\}\right)\right)+\ell_G(B_1)+\ell_G(B_2).
\end{align*}
Since we reach a contradiction in all of the cases we conclude that every pair of leaf block of $G-(\Le(G)\cap\Si_1(G))$ have a common cut vertex (see for instance the graph depicted in Fig.~\ref{fig: block4}) and thus every leaf block of $G-(\Le(G)\cap\Si_1(G))$ have the same common cut vertex in $G-(\Le(G)\cap\Si_1(G))$.
\begin{figure}
\centering
    \includegraphics[scale=0.6]{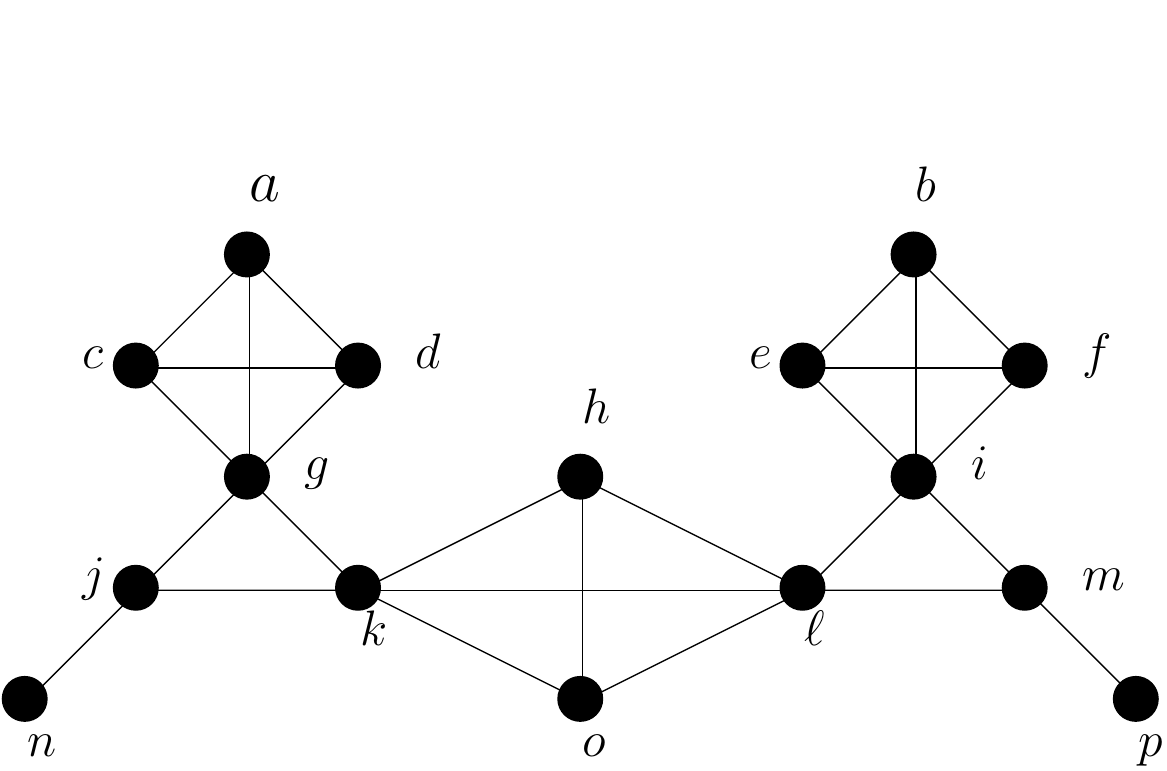}   
    \caption{Block graphs with four leaf blocks, one nonleaf simplicial blocks and two nonsimplicial blocks.}
    \label{fig: block3}
\end{figure}
\end{proof}

The \emph{pineapple} $P_q^p$ is the graph whose vertex set can be partitioned into a clique $Q$ on $q$ vertices and a stable set $I$ on $p$ vertices such that every vertex of $I$ is adjacent to the same vertex in $Q$ (see Fig~\ref{fig: pineapple}). 

\begin{thm}
Let $G$ be a block graph on $n$ vertices having maximum independence number $\alpha$.  Then, $\rho(G)\le \rho(P_{n-\alpha+1}^{\alpha-1})$. In addition, the equality holds if and only if $G=P_{n-\alpha+1}^{\alpha-1}$.
\end{thm}

\begin{proof}
Let $G$ be a block graph with maximum independence set $\alpha$. Assume that $\alpha\ge 2$, otherwise $G=P_{n}^0=K_n$. By Lemma~\ref{lem: bloques separados} either every block of $G-(\Le(G)\cap\Si_1(G))$ has common cut vertex $v$ (see Fig.~\ref{fig: block4} for an example), or $G-(\Le(G)\cap\Si_1(G))=K_r$. Let $b$ be the number of nonleaf blocks having at least one simplicial vertex, let $t$ be the number of leaf blocks sharing the cut vertex $v$, and let $\ell$ be the number of leaf blocks of $G$ such that $v$ is not in their vertex sets. By Lemma~\ref{lem: indpendent set} and Corollary~\ref{cor: leaf block formula}, $\alpha=\ell+1$ whenever $b=0$, $t=0$ and $G-(\Le(G)\cap\Si_1(G))\neq K_r$, or $\alpha=b+t+\ell$, otherwise. 

In the sequel, we transform $G$ into $G^*$, whose vertex sets agree, where $v$ either is the only cut vertex of $G^*$ in every nonleaf block of $G^*$ or is the only simplicial vertex of the only nonleaf block of $G^*$, we will use $b^*$ to denote the number of nonleaf blocks in $G^*$ having at least one simplicial vertex, and $t^*$ to denote the number of leaf blocks sharing the cut vertex $v$, and $\ell^*$ to denote the number of leaf blocks such that $v$ does not belong to them. We will split the proof into four claims.

\vspace{0.5cm}

\emph{Claim 1: There exists at most one nonleaf block in $G$ without simplicial vertices. }

\vspace{0.5cm}

Suppose, towards a contradiction, that there exist two nonoleaf blocks $B_1$ and $B_2$ without simplicial vertices. Consider the graph $G^*$ obtained from $G$ by adding every edge $v_1v_2$ with $v_i\in V(B_i)\setminus\{v\}$ for each $i\in\{1,2\}$. Clearly, $G^*$ is a block graph. On the one hand, if $b=0$ and $t=0$ then $b^*=1$, whenever there is exactly two nonleaf blocks sharing the cut vertex $v$, or else $b^*=0$, $t^*=t$ and $\ell^*=\ell$. On the other hand, $b^*=b$, $t^*=t$ and $\ell^*=\ell$. Hence $\alpha(G^*)=\alpha$. Besides, by Lemma~\ref{lem: adding edges} we have $\rho(G)<\rho(G^*)$. We reach a contradiction. Therefore, $G$ has at most one nonleaf block without simplicial vertices. 

\medskip

Claim 1 implies $\alpha=b+t+\ell$.

\begin{figure}
\centering
    \includegraphics[scale=0.6]{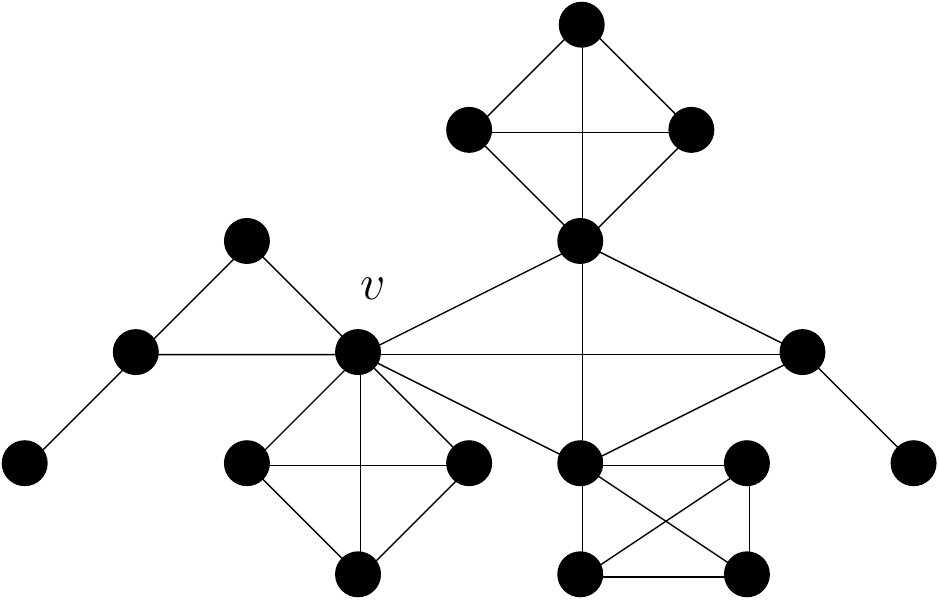}   
    \caption{Block graph $G$ with $G-(\Le(G)\cap\Si_1(G))$ having all its block sharing the cut vertex $v$.}
    \label{fig: block4}
\end{figure}

\vspace{0.5cm}

\emph{Claim 2: There exists at most one nonleaf block $B_1$}

\vspace{0.5cm}

Suppose, towards a contradiction, that there exist two nonleaf blocks $B_1$ and $B_2$ in $G$. By Claim 1 at most one of $B_1$ and $B_2$ have no simplicial vertex. First, assume, without lose of generality, that $V(B_2)$ contains no simplicial vertex. Set $S_1=V(B_1)\cap\Si_1(G)$. We transform the graph $G$ into $G^*$ by adding every edge $v_1v_2$ with $v_i\in V(B_1)\setminus\{v\}$ for every $i\in\{1,2\}$. Clearly, $b^*=b$, $\ell^*=\ell$ and $t^*=t$. Hence $\alpha(G^*)=\alpha$. In addition, by Lemma~\ref{lem: adding edges}, $\rho(G)<\rho(G^*)$, reaching a contradiction. Finally, assume that $V(B_i)\cap \Si_1(G)\neq\emptyset$  and let $S_i=V(B_i)\cap\Si_1(G)$, for each $i\in\{1,2\}$. Suppose, without losing generality, that $\sum_{u\in S_1} x_u\ge \sum_{u\in S_2} x_u$. We construct the graph $G^*$ from $G$ by deleting every edge $xy$ with $x\in S_2$ and $y\in V(B_2)\setminus(S_2\cup\{v\})$ and adding every edge $yz$ with $y\in V(B_2)\setminus(S_2\cup \{v\})$ and $z\in V(B_1)\setminus\{v\}$. Clearly, $b^*=b-1$, $t^*=t+1$ and $\ell^*=\ell$. Hence, $\alpha(G^*)=\alpha$. Besides, by Lemma~\ref{lem: k-trees}, $\rho(G)<\rho(G^*)$, reaching a contradiction.

\vspace{0.5cm}

\emph{Claim 3: Every block in $G$ is a leaf block. }

\vspace{0.5cm}

Suppose, towards a contradiction, that $G$ has at least a nonleaf block. First assume that $B$ (by Claim 1) is the only nonleaf block in $G$ having no simplicial vertex. Hence, by Claim 2, the remaining blocks are leaf blocks.  Let $B'$ be one of those leaf blocks having $v'$ as the only cut vertex of $B'$ in $G$. By Lemma~\ref{lem: adding edges}, the graph $G^*$ obtained from $G$ by adding every edge $ww'$ with $w\in V(B)\setminus\{v'\}$ and $w'\in V(B')\setminus\{v'\}$ satisfies $\rho(G)<\rho(G^*)$. In addition, $b^*=1$, $t^*=t$ and $\ell^*=\ell-1$. Hence, $\alpha(G^*)=\alpha$, reaching a contradiction. 

Assume now that every nonleaf block of $G$ has at least one simplicial vertex. By Claim 2 we conclude that $G$ has only one nonleaf block having at least one simplicial vertex. Hence there exists a leaf  block $B$ having $u$ as the only cut vertex of $G$. Suppose that $B'$ is another leaf block having $u'$ as the only cut vertex of $G$ with $u'\neq u$. Assume first that $x_u\ge x_{u'}$. By Lemma~\ref{lem: moving-vertices}, the graph $G^*$ obtained from $G$ by deleting every edge $w'u'$ with $w'\in V(B')\setminus\{u'\}$ and adding every edge $w'u$  with $w'\in V(B')\setminus\{u'\}$, satisfies $\rho(G)<\rho(G^*)$. Notice that, $b^*=b-1=0$, whenever $\ell=2$, and $b^*=b$ if $\ell>2$. In both cases $\alpha(G^*)=\alpha$. By symmetry, if $x_u\le x_{u'}$ applying the analogous transformation we obtain a graph $G^*$  with $\alpha(G^*)=\alpha$ such that $\rho(G)<\rho(G^*)$.  We reach a contradiction. 

\vspace{0.5cm}

\emph{Claim 4: $G=P_{n-\alpha+1}^{\alpha-1}$.}

\vspace{0.5cm}

Claim 3 implies that every block in $G$ is a leaf block, sharing a cut vertex $u$. Hence it remains to prove that at most one block $B$ has at least three vertices. Notice that if every leaf block in $G$ has exactly two vertices, then $G=K_{1,n-1}$. Suppose, towards a contradiction, that $B_1$ and $B_2$ are two leaf blocks having at least three vertices. Let $u_i\in V(B_i)$ such that $u_i\neq u$ and let $S_i=V(B_i)\setminus\{u,u_i\}$, for each $i\in\{1,2\}$. Assume, without losing generality, that $\sum_{s\in S_2}x_s\le\sum_{s\in S_1}x_s$. Hence, by Lemmas~\ref{lem: adding edges} and~\ref{lem: k-trees}, the graph $G^*$ obtained from $G$ by deleting every edge $u_2s$ with $s\in S_2$ and adding every edge $u_2w$ with $w\in S_1\cup\{v_1\}$, satisfies $\rho(G)<\rho(G^*)$. Besides, clearly $\alpha(G^*)=\alpha$, reaching a contradiction. \end{proof}

The following lemma give an upper bound of the spectral radius of the pinapple graph. 

\begin{figure}
\centering
    \includegraphics[scale=0.6]{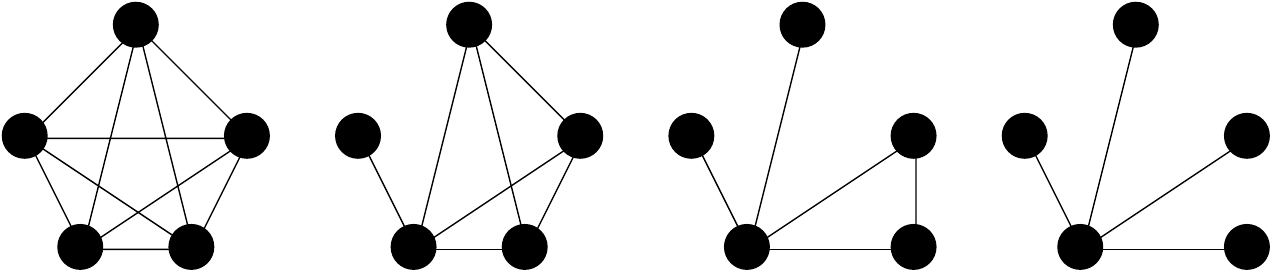}   
    \caption{From left to right we have $P_{5-\alpha+1}^{\alpha-1}$ for every $1\le\alpha\le 4$.}
    \label{fig: pineapple}
\end{figure}




\begin{lem}\label{radius_pineapple}
	Let $P_{n-\alpha+1}^{\alpha-1}$ be the pineapple graph with $2 \leq \alpha \le n-2$. Then 
	\begin{equation}\label{bound radius pineapple 1}
	\rho(P_{n-\alpha+1}^{\alpha-1}) \leq 
	\beta - 1 + \frac{\sqrt{(\beta^2-n)^2+ 4 (n-\beta)(2\beta -1)} - (\beta^2-n)}{4\beta-2}, 
	\end{equation}
	for $2 \leq \alpha \le n-\sqrt{n-1}$, and
	\begin{equation}\label{bound radius pineapple 2}
	\rho(P_{n-\alpha+1}^{\alpha-1}) \leq 
	\frac{ 2 \sqrt{n - 1} + \sqrt{(\alpha-1)\gamma^2+ (n-\alpha)(2 -\gamma)}}{2 +\gamma},
	\end{equation} 
	for $n-\sqrt{n-1} < \alpha \leq n-2$, where $\beta=n-\alpha+1$ and $\gamma=1 - \frac{n-\alpha-1}{\sqrt{n-1}}$.
\end{lem}
\begin{proof}
	In~\cite[Proposition 1.1]{TSH2016}, it is proved that the characteristic polynomial  $p(x)=\det(xI-A)$, where $A$ is the adjacency matrix of $P_{n-\alpha+1}^{\alpha-1}$, satisfies
	\[
	p(x) = x^{\alpha-2} (x+1)^{n-\alpha-1} (x^3 - (n-\alpha-1)x^2 - (n-1)x + (\alpha-1) (n-\alpha-1)).
	\]
 Perron-Frobenius implies that $\rho(P_{n-\alpha+1}^{\alpha-1})$ coincides with the maximum positive root of  
	\begin{equation}\label{eq_pineapple}
	q(x) = x^3 -(n-\alpha-1)x^2 -(n-1)x + (\alpha-1) (n-\alpha-1).
	\end{equation}
	We will find an upper bound to the maximum positive root of $q$. Notice that the pineapple $P_{n-\alpha+1}^{\alpha-1}$ contains $K_{n-\alpha+1}$ and  $K_{1,n-1}$ as a subgraph, based on this fact $\max\{n-\alpha,\sqrt{n-1}\}\le\rho(P_{n-\alpha+1}^{\alpha-1})$ (see \cite[Corollary 7]{LLT2004} for more details). We will split the task into two cases. 
	
\vspace{0.5cm}
	\emph{Case 1: $2 \leq \alpha \le n-\sqrt{n-1}$.} 
\vspace{0.5cm}	

It is easy to see that
\[
	\rho(P_{n-\alpha+1}^{\alpha-1}) =  \beta - 1 + t,
\] 
	where $\beta=n-\alpha+1$ and $t$ is the maximum positive solution of
	\[
	x^3 + (2\beta-1)x^2 + (\beta^2-n)x - (\alpha-1)=0.
	\]
	Since $t>0$, we have that 
	\[
	(2\beta-1)t^2 + (\beta^2-n)t - (\alpha-1) < 0.
	\]
	It follows immediately that
	\[
	t \le \frac{\sqrt{(\beta^2-n)^2+ 4 (n-\beta)(2\beta -1)} - (\beta^2-n)}{4\beta-2}.
	\]
	Finally, we conclude 
	\[
	\rho(P_{n-\alpha+1}^{\alpha-1}) \le \beta - 1 + \frac{\sqrt{(\beta^2-n)^2+ 4 (n-\beta)(2\beta -1)} - (\beta^2-n)}{4\beta-2}.
	\]
	
\vspace{0.5cm}	
	\emph{Case 2: $n-\sqrt{n-1} < \alpha \le n-2$.} 
\vspace{0.5cm}	

It is easy to see that
	\[
	\rho(P_{n-\alpha+1}^{\alpha-1}) =  \sqrt{n - 1} + t,
	\] 
	where $t$ is the maximum positive solution of
	\[
	x^3 + \sqrt{n-1}(2 +\gamma) \ x^2 + 2(n-1)\gamma \ x - (n-\alpha)(n-\alpha-1)=0,
	\]
	where $\gamma=1 - \frac{n-\alpha-1}{\sqrt{n-1}}$. Since $t>0$, we see that 
	\[
	\sqrt{n-1}(2 +\gamma) \ t^2 + 2(n-1)\gamma \ t - (n-\alpha)(n-\alpha-1) < 0.
	\]
	It follows immediately that
	\[
	t \le \frac{\sqrt{(\alpha-1)\gamma^2+ (n-\alpha)(2 -\gamma)} - \sqrt{n-1}\gamma}{2 +\gamma}.
	\]
	Finally, we conclude 
	\begin{eqnarray}
	\label{bound radius pineapple 2 bis} \rho(P_{n-\alpha+1}^{\alpha-1}) &\le& \sqrt{n - 1} + \frac{\sqrt{(\alpha-1)\gamma^2+ (n-\alpha)(2 -\gamma)} - 
		\sqrt{n-1}\gamma}{2 +\gamma}\\
	\nonumber                                && = \frac{ 2 \sqrt{n - 1} + \sqrt{(\alpha-1)\gamma^2+ (n-\alpha)(2 -\gamma)}}{2 +\gamma}.
	\end{eqnarray}
\end{proof}

\begin{rmk}
	In this remark, we compare the bounds for the spectral radius $\rho(P_{n-\alpha+1}^{\alpha-1})$ obtained in Lemma \ref{radius_pineapple} with the bounds in 
	\cite[Corollaries 7 and 8]{LLT2004}.
	
	Under the assumption $2 \leq \alpha \le n-\sqrt{n-1}$, we have 
	$$
	\rho(P_{n-\alpha+1}^{\alpha-1}) \le \beta - 1 + \frac{\sqrt{(\beta^2-n)^2+ 4 (n-\beta)(2\beta -1)} - (\beta^2-n)}{4\beta-2}.
	$$
	By the Mean Value Theorem, we see that
	$$
	\frac{\sqrt{(\beta^2-n)^2+ 4 (n-\beta)(2\beta -1)} - (\beta^2-n)}{4\beta-2} =  \frac{4 
		(n-\beta)(2\beta -1)}{(4\beta-2) 2 \sqrt{\xi}} = \frac{(n-\beta)}{\sqrt{\xi}},
	$$
	where $(\beta^2-n)^2 < \xi <  (\beta^2-n)^2+ 4 (n-\beta)(2\beta -1) $. It follows that 
	$$
	\beta - 1 + \frac{\sqrt{(\beta^2-n)^2+ 4 (n-\beta)(2\beta -1)} - (\beta^2-n)}{4\beta-2} < \beta - 1 + \frac{(n-\beta)}{\beta^2-n}. 
	$$
	Hence the bound \eqref{bound radius pineapple 1} refines the one present in \cite[Corollaries 7]{LLT2004}.
	
	We now turn to the case $n-\sqrt{n-1} < \alpha \le n-2$. By \eqref{bound radius pineapple 2 bis}, we have 
	$$
	\rho(P_{n-\alpha+1}^{\alpha-1}) \le \sqrt{n - 1} + \frac{\sqrt{(\alpha-1)\gamma^2+ (n-\alpha)(2 -\gamma)} - 
		\sqrt{n-1}\gamma}{2 +\gamma}
	$$
	By the Mean Value Theorem, we see that
	$$
	\frac{\sqrt{(\alpha-1)\gamma^2+ (n-\alpha)(2 -\gamma)} - \sqrt{n-1}\gamma}{2 +\gamma} =  \frac{(n-\alpha)(1-\gamma)}{ 2 \sqrt{\xi}},
	$$
	where $(n-1)\gamma^2 < \xi <  (\alpha-1)\gamma^2+ (n-\alpha)(2 -\gamma) $. It follows that 
	$$
	\sqrt{n - 1} + \frac{\sqrt{(\alpha-1)\gamma^2+ (n-\alpha)(2 -\gamma)} - \sqrt{n-1}\gamma}{2 +\gamma} < \sqrt{n - 1} + \frac{(n-\alpha)(1-\gamma)}{ 2 
		\sqrt{n-1} \gamma}. 
	$$
	Thus the bound \eqref{bound radius pineapple 2} refines the one presented	 in \cite[Corollaries 8]{LLT2004}.
\end{rmk}

\begin{cor}
	Let $G$ be a block graph on $n$ vertices having maximum independence number $\alpha$.  Then,  
	\begin{equation}
	\rho(G)\le  
	\beta - 1 + \frac{\sqrt{(\beta^2-n)^2+ 4 (n-\beta)(2\beta -1)} - (\beta^2-n)}{4\beta-2}, 
	\end{equation}
	for $2 \leq \alpha \le n-\sqrt{n-1}$, and
	\begin{equation}
	\rho(G) \le 
	\frac{ 2 \sqrt{n - 1} + \sqrt{(\alpha-1)\gamma^2+ (n-\alpha)(2 -\gamma)}}{2 +\gamma},
	\end{equation} 
	for $n-\sqrt{n-1} < \alpha \leq n-2$, where $\beta=n-\alpha+1$ and $\gamma=1 - \frac{n-\alpha-1}{\sqrt{n-1}}$.
\end{cor}

\bibliographystyle{abbrv}
\bibliography{spectralradius}

\end{document}